\documentclass[12pt,english,reqno]{amsart}
\usepackage[latin9]{inputenc}
\usepackage{verbatim}
\usepackage{textcomp}
\usepackage{amsthm}
\usepackage{amstext}
\usepackage{amssymb}
\usepackage{graphicx}
\usepackage{esint}
\date{\today}

\makeatletter
\newcounter{cprop}[section]

\usepackage{amscd}\usepackage{amsthm}\usepackage[english]{babel}
\usepackage[arrow,matrix]{xy}\usepackage{cite}\usepackage{bbm}
\@ifundefined{definecolor}
 {\usepackage{color}}{}
\usepackage{MnSymbol}

\theoremstyle{definition}
\newtheorem{thm}[cprop]{Theorem}
\newtheorem{rem}[cprop]{Remark}
\newtheorem{lemma}[cprop]{Lemma}
\newtheorem{pro}[cprop]{Proposition}

\newtheorem{rk}[cprop]{Remark}
\newtheorem{defn}[cprop]{Definition}
\newtheorem{ex}[cprop]{Example}

\newcommand{\blue}{}
\newcommand{\red}[1]{}
\numberwithin{equation}{section} 
\newcommand{\td}{\tilde}
\newcommand{\OO}{{\mathcal O}}

\newcommand{\F}{{\mathcal F}}

\newcommand{\LL}{{\mathcal L}}
\newcommand{\EE}{{\mathcal E}}
\newcommand{\E}{{\mathbb E}}

\newcommand{\w}{{\bf w}}
\newcommand{\s}{{\sigma}}

\newcommand{\bea}{\begin{eqnarray}}
\newcommand{\eea}{\end{eqnarray}}

\newcommand{\N}{\mathbb N}

\newcommand{\R}{\mathbb{R}}

\newcommand{\PP}{\mathbb{P}}
\newcommand{\FF}{\mathbb{F}}

\newcommand{\<}{\langle}
\renewcommand{\>}{\rangle}
\def\d {\triangle}
\def\t {\theta}
\def\w {\omega}
\def\0 {\emptyset}

\def\a {\alpha}
\def\l {\lambda}
\def\d {\delta}
\def\e {\varepsilon}

\def\b {\beta}

\def\s {\sigma}

\newcommand{\diam}{\mathrm{diam}}

\makeatother

\usepackage{babel}

\begin{document}
\title[Synchronization by noise for order-preserving RDS]{Synchronization by noise for order-preserving random dynamical systems}

\begin{abstract}
We provide sufficient conditions for weak synchronization/{\blue stabilization} by noise for order-preserving random dynamical systems on Polish spaces. That is, under these conditions we prove the existence of a weak point attractor consisting of a single random point. This generalizes previous results in two directions: First, we do not restrict to Banach spaces and second, we do not require the partial order to be admissible nor normal. As a second main result and application we prove weak synchronization by noise for stochastic porous media equations with additive noise.
\end{abstract}

\author[F. Flandoli]{Franco Flandoli}
\address{Dipartimento di Matematica\\
Largo Bruno Pontecorvo 5\\
56127 Pisa\\
Italy}
\email{flandoli@dma.unipi.it}

\author[B. Gess]{Benjamin Gess}
\address{Max-Planck Institute for Mathematics in the Sciences \\
Inselstra\ss e 22\\
04103 Leipzig\\
Germany}
\email{bgess@mis.mpg.de}

\author[M. Scheutzow]{Michael Scheutzow}
\address{Institut f\"ur Mathematik, MA 7-5\\
Technische Universit\"at Berlin\\
10623 Berlin \\
Germany}
\email{ms@math.tu-berlin.de}



\thanks{B.G. has been partially supported by the research project ``Random dynamical systems and regularization by noise for stochastic partial differential equations'' funded by the German Research Foundation.}

\maketitle

\section{Introduction}

In this work we provide sufficient conditions for (weak) synchronization by noise for strongly mixing, order-preserving random dynamical systems\footnote{For notation and background on RDS see Section \ref{sec:notation} and Appendix \ref{sec:RDS} below} (RDS) $\varphi$ on partially ordered Polish spaces $(E,d)$. Weak synchronization by noise here means that there is a weak point attractor consisting of a single random point and in this sense the random dynamics are asymptotically globally stable. In particular, in this case
\begin{equation}\label{eq:traj_conv}
  d(\varphi_t(\omega,x),\varphi_t(\omega,y))\to 0,\quad\text{ for }t\to\infty
\end{equation}
in probability, for all $x,y\in E$.

More precisely, assuming a concentration property for the corresponding invariant measure $\mu$ on intervals in $E$ (cf.\ \eqref{eq:intro_1} below), we prove the existence of a unique $\varphi$-invariant random point $a:\Omega\to E$, measurable with respect to the past $\F_0$, such that 
  $$d(\varphi_t(\omega,x),a(\t_{t}\omega))\to 0,\quad\text{ for }t\to\infty$$
in probability, for all $x\in E$. The method of proof is entirely new. Several examples illustrating the generality of this result are presented in Section \ref{sec:examples}.

As a second main result we prove weak synchronization by noise for stochastic porous media equations of the type
\begin{equation}\label{eq:SPME:_intro}
  dX_t = \left(\Delta X_t^{[m]}+X_t\right)dt+dW_t,
\end{equation}
with zero Dirichlet boundary conditions on bounded, smooth domains $\OO\subseteq \R^d$, $d\le 4$, $m > 1$ and $W$ being a trace-class Wiener process satisfying an appropriate non-degeneracy condition. Here we use the convention $u^{[m]}:=|u|^{m-1}u$. This solves a problem left open in \cite{G13}. In contrast, the attractor for the deterministic porous medium equation 
$$
  dX_t = \left(\Delta X_t^{[m]}+X_t\right)dt
$$
has infinite fractal dimension (cf.\ \cite{EZ08}). We prove that this infinite dimensional attractor collapses into a zero dimensional random attractor if sufficiently non-degenerate noise is added.

Our results on order-preserving RDS generalize those of \cite{CS04} in two main directions: First, we do not require the underlying space $E$ to be embedded in a (partially ordered) Banach space. Second, we completely remove the assumptions on the partial order to be ``admissible" and normal. More precisely, in \cite{CS04} it is required that the RDS $\varphi$ is defined on an admissible subset $E$ of a real, separable Banach space $V$. Admissibility here means, in particular, that for each compact set $K \subseteq E$ there are $a,b\in V$ such that $K\subseteq \textrm{int}_E([a,b]\cap E)$. In infinite dimensions this is a restrictive condition since intervals $[a,b]$ may have empty interior and, even worse, compact sets are not necessarily included in intervals (e.g.\ consider $L^p$ spaces). Therefore, in applications to SPDE one typically has to choose $E$ to be the set of continuous functions, thus restricting to SPDE for which spatial continuity of solutions can be shown. This often leads to stringent restrictions on the spatial dimension or to assumptions on the spatial regularity of the noise. In this paper, we replace the assumption of admissibility by a support assumption on the invariant measure $\mu$, i.e.\ we assume that for each $\e>0$ there is an interval $[f,g] \subseteq E$ such that
\begin{equation}\label{eq:intro_1}
  \mu([f,g])\ge 1-\e.
\end{equation}
The advantage is that the invariant measure $\mu$ often has support on smaller spaces than all of $E$ and thus in applications this support condition can be seen to be satisfied even though admissibility is not. \\
In order to have admissibility of a partial order, or more generally \eqref{eq:intro_1}, one wants intervals $[f,g]$ to be ``large''. On the other hand, normality of a partial order (in Banach spaces $E$ say) requires the existence of a constant $C>0$ such that $\diam([f,g])\le C\|f-g\|_E$ for all intervals $[f,g]$. Hence, in order for a partial order to be normal intervals may not be ``too large''. In this sense, admissibility (or \eqref{eq:intro_1} resp.) and normality are conflicting assumptions limiting the applicability to SPDE, which explains the relevance of removing the normality assumption. 

In particular, these generalizations are crucial in their application to weak synchronization by noise for \eqref{eq:SPME:_intro}. This was left as an open problem in \cite{G13}, since the usual partial order ``$\le$" on $E:=H^{-1}=(H_0^1)^*$ is not admissible. In addition, ergodicity for \eqref{eq:SPME:_intro} is known only in cases of non-degenerate noise, for which there is in general no hope to prove spatial continuity of solutions. Moreover, also \eqref{eq:intro_1} is unclear for the usual partial order ``$\le$". The main idea here is to introduce an alternative, non-standard partial order ``$\preceq$" on $H^{-1}$, for which \eqref{eq:intro_1} can be proven. Indeed, intervals with respect to ``$\preceq$" can be seen to be much larger than those corresponding to ``$\le$". On the downside, this causes ``$\preceq$" to be not normal (cf.\ the discussion above). In conclusion, the non-standard partial order ``$\preceq$" is neither normal nor admissible, thus requiring the full generality of our first main result.  

Let us now briefly comment on the existing literature, for more details we refer to \cite{FGS14}. Synchronization by noise for order-preserving RDS has been analyzed, for example, in \cite{AC98,CS04,C02} and was first applied to prove synchronization for stochastic reaction-diffusion systems on thin two-layer domains in \cite{CCK07}. Methods based on local stability have been introduced in \cite{B91} and large deviation techniques have been employed in  \cite{MS88,MSS94,T08}. Synchronization by noise for SPDE has been investigated, for example, in \cite{ACW83,CR04,CCLR07,G13}. For the related effect of synchronization in master-slave systems we refer to \cite{C10} and the references therein. For synchronization for discrete time RDS see \cite{H13, N14,KJR13,JK13} and the references therein. Applications of synchronization by noise are to be found, for example, in theoretical physics \cite{RSTA95,P96,PRKH02,KJR13}, climate dynamics \cite{GCS08,CSG11,C13}, neurophysiology \cite{SSV14} and numerics \cite{LPP13}.

{\blue Concerning the terminology of synchronization by noise, different and somewhat inconsistent terminology has been used in the literature. In some instances, the effect that deterministic invariant points may become stable due to the inclusion of noise has been referred to as stabilization by noise (e.g.\ \cite{ACH83,CR04,CCLR07,A02,CKS06}). In these examples, the deterministic and stochastic systems share the same deterministic invariant points. The property that each two trajectories of a noisy system converge to one another, i.e.\ \eqref{eq:traj_conv} holds, has been named synchronization by noise in several recent publications (e.g.\ \cite{H13,N14}). This property is closely related and a simple consequence of the results obtained in this work. We therefore use the notion of synchronization by noise, noting, however, that there would be good reason to refer to the effects observed here as stabilization by noise.
}

Outline of the paper: In Section \ref{sec:main_results} we prove synchronization by noise for general order-preserving RDS, in Section \ref{sec:SPME} for stochastic porous media equations. Further applications to stochastic differential inclusions and SPDE with two reflecting walls are presented in Section \ref{sec:examples}.

\subsection{Notation}\label{sec:notation}
For a set $A\subseteq E$ we let $\diam_E(A):=\sup_{a,b\in A}d(a,b)$, $A^c$ denotes its complement and $B_\d(A) := \{x\in E:\ d(x,A)=\inf_{a\in A}d(x,a) < \d\}$. For simplicity we often suppress the notation of $E$ and write $\diam(A)$ instead. A subset $X \subseteq E$ is said to be admissible, if $X$ is a Polish space in $E$ and for every compact set $K\subseteq X$ there are $a,b\in E$, $a\le b$ such that $K\subseteq \mathrm{int}_X([a,b]\cap X)$. 

We let $(\Omega,\F,\PP)$ be a probability space. For a random variable $v:\Omega\to E$ we let $\LL(v):=v_*\PP$ be its law. For $f,g\in E$ with $f\le g$ we define $[f,g]_{E,\le} := \{x \in E:\ f\le x\le g\}$. If the partial order ``$\le$" or underlying space $E$ are clear from the context, we write $[f,g]_\le,[f,g]_E$ or $[f,g]$ instead. For a sequence of sets $A_n$ we set $\{A_n\ \text{i.o.}\}:=\{x\in\bigcup_{n\in\N} A_n:\ x\in A_n\text{ for infinitely many } n \in \N \}$.

\section{Order preserving random dynamical systems}\label{sec:main_results}

Let $(E,d)$ be a Polish space with partial order ``$\le$" such that 
\begin{equation}\label{eq:M}
  M:=\{(x,y) \in E \times E: x \le y\}
\end{equation}
is closed in $E\times E$ (cf. e.g. \cite[p. 128]{L92}, \cite{KKOB77}). Equivalently, from $x_n,y_n \in E$ with $x_n \le y_n$ and $x_n\to x,y_n\to y$ it follows $x\le y$.

\begin{defn}\label{def:normal}
  We say that the partial order of $E$ is \textit{normal} if there is a function $h:\R_+\to \R_+ \cup \{+\infty\}$ satisfying $\lim_{t\downarrow 0} h(t)=0$ such that
  \begin{equation}\label{eqn:normal}
     \diam([f,g])  \le h(d(f,g)),
  \end{equation}
  for each $f \le g$, where $[f,g]=\{x\in E:f\le x \le g\}$.
\end{defn}

{\blue
The notion of a normal partial order introduced above extends the well-known notion of a normal partial order on a Banach space. Indeed, first recall that a partial order on a Banach space $(E,\|\cdot\|_E)$ is said to be normal if there is a constant $C>0$ such that for all $0\le f\le g$ one has $\|f\|_E\le C\|g\|_E$ (cf.\ e.g.\ \cite{CS04,C02,A76,K64}). This is easily seen to be equivalent to the existence of some constant $\widetilde C>0$ such that 
\begin{equation}\label{normalequiv}
{\mathrm{diam}} ([f,g])\le \widetilde C \|f-g\|_E
\end{equation} 
for all $f \le g$ and thus \eqref{eqn:normal} is satisfied.
%
%
%
%
%
%
Conversely, assume that \eqref{eqn:normal} holds 
and choose $\alpha>0$ such that $h(\alpha)<\infty$. Take $f,g \in E$ such that $f \le g$, $f\ne g$. Then
$$
{\mathrm{diam}} ([f,g]) =\frac 1{\alpha}\|f-g\|_E {\mathrm{diam}} \big(\big[\frac{\alpha}{\|f-g\|_E}  f,\frac{\alpha}{\|f-g\|_E} g\big]\big)\le \frac {h(\alpha)}{\alpha}\|f-g\|_E,
$$ 
so we obtain \eqref{normalequiv} with $\widetilde C=  \frac{h(\alpha)}{\alpha}$. Hence, the two concepts of normality coincide on a partially ordered Banach space.}

{\blue
\begin{rem}\label{rem:normal_2}
  A partial order ``$\le$'' is normal if and only if for each $\d>0$ there is an $\e>0$ such that for all $f \le g$ with $d(f,g)\le \e$ we have $d(a,b)\le \d,$ for all $a,b\in [f,g]$.
\end{rem}
\begin{proof}
  We only have to show that the condition in the statement implies normality. To see this, for $\e>0$ set
    $$h(\e):= \sup\{\diam([f,g]): f,g \in E, f\le g, d(f,g) \le \e\}$$
  and $h(0):=0$. Then $h(\e): \R_+ \to \R_+ \cup \{\infty\}$ is non-decreasing and for all $f,g \in E, f\le g$ we have $\diam([f,g]) \le h(d(f,g))$. Let $\d >0$. By assumption there is an $\e>0$ such that for all $f \le g$ with $d(f,g)\le \e$ we have $d(a,b)\le \d,$ for all $a,b\in [f,g]$. Thus, $\diam([f,g])\le \d$ and, hence, $h(\e) \le \d$, which yields $\lim_{\e \downarrow 0} h(\e) = 0$. 
\end{proof}
}
{\blue The generalization of the concept of normality of a partial order to Polish spaces will turn out to be crucial, the key point being the following proposition.}
\begin{pro}\label{prop:normal}
   Let $K\subseteq E$ be a compact set. Then, $(K,d)$ is a Polish space with normal partial order ``$\le$''.  
\end{pro}
\begin{proof}
Assume that  ``$\le$'' is not normal. Then there is a $\d>0$ such that for all $\e>0$ there are $f^\e \le g^\e$ with $d(f^\e,g^\e)\le \e$, $f^\e,g^\e \in K$ and $a^\e,b^\e \in [f^\e,g^\e]\subseteq K$ such that $d(a^\e,b^\e) \ge \d$. By compactness of $K$ we may choose a sequence $\e_n\to 0$ such that $f^{\e_n},g^{\e_n},a^{\e_n},b^{\e_n}\to f,g,a,b$ respectively. Since $d(f^\e,g^\e)\le \e$ we have $f=g$. Moreover, since $a^\e,b^\e \in [f^\e,g^\e]$ we have $a=b=f$ since $M$ in \eqref{eq:M} is closed. In particular, $a^{\e_n},b^{\e_n}\to a$ in contradiction to $d(a^\e,b^\e) \ge \d$. {\blue By Remark \ref{rem:normal_2} this proves normality of ``$\le$''}.
\end{proof}

The following proposition generalizes \cite[Proposition 1]{CS04}, which required $E$ to be embedded into a partially ordered Banach space $V$, by removing this embedding condition. Note that the proof in \cite{CS04} relies on the linear structure of $E$ and thus the proof given here is significantly different. 
\begin{pro}\label{prop:monotone_convergence}
   Let $X_t,Y_t$ be two stochastic processes taking values in $E$, satisfying $X_t(\omega)\le Y_t(\omega)$ for all $t\in\R_+$, $\omega\in\Omega$. Further, assume that the laws $\LL(X_t),\LL(Y_t)$ converge weakly$^*$ to $\mu$ for $t\to\infty.$ Then,
     $$d(X_t,Y_t)\to 0\quad \text{for }  t\to\infty,$$
   in probability.
\end{pro}
\begin{proof}
  \textit{Step 1:} Consider the joint distribution
    $$\pi_t = \LL(X_t,Y_t).$$
  Since $\LL(X_t),\LL(Y_t)$ converge weakly$^*$ to $\mu$, $\{\pi_t\}_{t\ge t_0}$ is tight for some $t_0 \ge 0$. Moreover, $\pi_t(M)=1$, where $M$ is given in \eqref{eq:M}. Hence, we may extract a subsequence $(t_n) \to \infty$ such that 
    $$\pi_{t_n} \rightharpoonup \pi\quad\text{weakly}^*$$
  and $ \pi(M)=1$ (since $M$ is closed). Moreover, both marginals of $\pi$ are equal to $\mu$.
    
  \textit{Step 2:} We now prove that $\pi$ is necessarily concentrated on the diagonal.
  
  Assume the contrary. Then there exist $a<b$ such that $(a,b)$ is in the support of $\pi$. Since $M^c$ is open, there exists an open neighborhood $U$ of $(b,a)$ contained in $M^c$ and we may assume that $U$ is a rectangle, i.e. $U=U_b\times U_a$ for $U_a,U_b\subseteq E$ being open sets. Then $U':=U_a\times U_b$ is an open neighborhood of $(a,b)$ and $\pi(U')>0$ by definition of the support of $\pi$.  We can find compact subsets $K_a \subseteq U_a$ and $K_b\subseteq U_b$ such that $\pi(K_a \times K_b)>0$. Now we define
  \begin{align*}
     A:=\{x \in E: x \ge v \text{ for some } v \in K_b\}
  \end{align*}      
  The set $A$ is closed, since $M$ is closed and $K_b$ is compact. Therefore $A$ is Borel. Moreover, by definition, $A$ is an increasing set in the sense that $x \in A$ and $y \ge x$ implies $y \in A$. Furthermore, $A$ and $K_a$ are disjoint. {\blue  Indeed, if $x \in A \cap K_a$, then $x \in A$ implies that there exists some $v \in K_b$ such that $x \ge v$ but $(x,v)\in K_a \times K_b$ 
  implies $x<v$ which is a contradiction. We further note that the indicator function $f$ of $A$ is measurable and non-decreasing, that is  $x \le y$ implies $f(x) \le f(y) $, since $A$ is an increasing set.} We conclude that for $(X,Y)$ being a random variable with law $\pi$ we have
  \begin{align*}
     \E f(X) 
     &= \E f(X)1_{K_a \times K_b}(X,Y)+\E f(X)1_{(K_a \times K_b)^c}(X,Y)\\
     &\le \E f(X)1_{K_a \times K_b}(X,Y)+\E f(Y)1_{(K_a \times K_b)^c}(X,Y),
  \end{align*}    
  {\blue since $\pi$ is concentrated on $M$.}  
  Moreover,
    $$ \E f(Y) 1_{K_a \times K_b}(X,Y)>0=\E f(X)1_{K_a \times K_b}(X,Y)$$
  since $\pi(K_a \times K_b)>0$ and $A \supseteq K_b$. Hence,
    \begin{align*}
       \E f(X) < \E f(Y)
    \end{align*}
  in contradiction to $\LL(X)=\LL(Y)$.
  
  \textit{Step 3:} Since $\pi$ is concentrated on the diagonal and has marginals $\mu$, $\pi$ is the image measure of $\mu$ under the map $x\mapsto (x,x)$. In particular, the whole sequence $\pi_t$ converges to $\pi$. Thus,
    $$\E[d(X_t,Y_t)\wedge 1] \to  0,$$
  for $t\to\infty$ and thus
    $$\PP[d(X_t,Y_t)>\e] 
       \le \frac{1}{\e}\E [d(X_t,Y_t)\wedge 1]
       \to  0,$$
  for $t\to\infty$ and all $\e\in (0,1]$.
\end{proof}

To motivate the following Lemma, we recall that if $\varphi$ is a white noise RDS with associated Markovian semigroup $P_t f(x):=\E f(\varphi_t(\cdot,x))$ having $\mu$ as an invariant probability measure, then there exists a $\varphi$-invariant random probability measure $\pi_\cdot$, the so-called statistical equilibrium, obtained from $\mu$ via{\blue
  $$\pi_\omega = \lim_{k\to\infty} \varphi_{t_k}(\t_{-t_k}\omega)_* \mu\quad \PP-\text{a.s.},$$
where $t_k$ is an arbitrary sequence with $t_k\to\infty$} and one has $\E\pi_\cdot = \mu$. If $\varphi$ is not a white noise RDS then this construction fails and it is an open question in the literature how to define the statistical equilibrium, or to construct any $\varphi$-invariant random probability measure in this case. This is the purpose of the following lemma.

\begin{lemma}\label{lem:ex_stat_eq}
   Let $\varphi$ be a weakly mixing RDS with limit distribution $\mu$, i.e.\ for $\mu$-a.a.\ $x\in E$ we have $\LL(\varphi_t(\cdot,x))\rightharpoonup \mu$ weakly$^*$. Then there exists an $\F_0$-measurable, $\varphi$-invariant random probability measure $\pi_\cdot$ satisfying $\E\pi_\cdot = \mu$.
\end{lemma}
\begin{proof}
  We consider the random measures
  \[
  \pi_{\omega}^{t}:=\frac{1}{t}\int_{0}^{t}\varphi_{r}(\t_{-r}\omega)_{*}\mu dr
  \]
  and their averages 
  \begin{align*}
    \mu^{t}&:=\frac{1}{t}\E\int_{0}^{t}\varphi_{r}(\t_{-r}\cdot)_{*}\mu dr.
  \end{align*}
  Since $\varphi$ is weakly mixing, we have
  \begin{align*}
     \mu^{t}(f)&=\frac{1}{t}\int_{0}^{t}\int_E \E f(\varphi_{r}(\t_{-r}\cdot,x))d\mu(x) dr\\
     &\to \mu(f),\quad\text{for } t\to\infty,
  \end{align*}
  for each bounded, continuous $f:E\to \R$. Hence, there is a $t_0 \ge 0$ such that for each $\e>0$ there is a compact set $K^{\e}$ such that 
  \[
  \E\pi_{\omega}^{t}(K^{\e})=\mu^{t}(K^{\e})\ge1-\e,
  \]
  for all $t\ge t_0$. Consequently, the random measures $\pi^{t}_\cdot$ are tight (cf.\ \cite[Definition 4.2]{C02-2}) and thus (cf.\ \cite[Theorem 4.4]{C02-2}) there is a sequence $t_{n}\to\infty$ and a random measure $\pi_\cdot$ such that
  \[
  \pi^{t_{n}}_\cdot\rightharpoonup\pi_\cdot \quad\text{weakly}^* \text{ for } n\to \infty,
  \]
  i.e.\ for each random continuous function, that is each $f:\Omega\times E\to \R$ such that $\omega \mapsto f(\omega,x)$ is measurable for each $x\in E$, $x\mapsto f(\omega,x)$ is continuous and bounded for each $\omega\in \Omega$ and $\|f(\cdot,\cdot)\|_{L^1(\Omega;C_b(E))}<\infty$, we have
  \[
    \E\int_E f(\omega,x)d\pi^{t_{n}}_\omega(x)\to \E\int_E f(\omega,x)d\pi_\omega(x) \quad\text{for } n\to \infty.
    \]
  In particular, choosing $f$ independent of $\omega$ yields
  \[
      \mu^{t_n}(f)=E\int_E f(x)d\pi^{t_{n}}_\omega(x)\to \E\int_E f(x)d\pi_\omega(x) \quad\text{for } n\to \infty
      \]
  and thus $\E\pi_\cdot = \mu$.
  
  It remains to prove that $\pi$ is $\varphi$-invariant. We note that for all random continuous functions $f$ and all $t\ge 0$
  \begin{align*}
   & \E\int_E f(\omega,x)d\varphi_{t}(\omega)_{*}\pi_{\omega}(x)\\
   & =\E\int_E f(\omega,\varphi_{t}(\omega,x))d\pi_{\omega}(x)\\
   & =\lim_{n\to\infty}\frac{1}{t_{n}}\int_{0}^{t_{n}}\E\int_E f(\omega,\varphi_{t}(\omega,x))d\varphi_{r}(\t_{-r}\omega)_{*}\mu(x)dr\\
   & =\lim_{n\to\infty}\frac{1}{t_{n}}\int_{0}^{t_{n}}\E\int_E f(\omega,\varphi_{t}(\omega,\varphi_{r}(\t_{-r}\omega,x)))d\mu(x)dr\\
   & =\lim_{n\to\infty}\frac{1}{t_{n}}\int_{0}^{t_{n}}\E\int_E f(\omega,\varphi_{t+r}(\t_{-r}\omega,x))d\mu(x)dr\\
   & =\lim_{n\to\infty}\frac{1}{t_{n}}\int_{0}^{t_{n}+t}\E\int_E f(\omega,\varphi_{r}(\t_{-r+t}\omega,x))d\mu(x)dr\\
   & =\lim_{n\to\infty}\frac{1}{t_{n}}\int_{0}^{t_{n}}\E\int_E f(\t_{-t}\omega,x)d\varphi_{r}(\t_{-r}\omega)_{*}\mu(x)dr\\
   & =\E\int_E f(\t_{-t}\omega,x)d\pi_{\omega}(x)\\
   & =\E\int_E f(\omega,x)d\pi_{\t_{t}\omega}(x)
  \end{align*}
  and thus
  \[
  \varphi_{t}(\omega)_{*}\pi_{\omega}=\pi_{\t_{t}\omega}\quad \PP-\text{a.s..}
  \]
\end{proof}

\begin{thm}\label{thm:monotone}
  Let $\varphi$ be an order-preserving, strongly mixing\footnote{See Appendix \ref{sec:RDS} for the definition.} RDS on $E$ with limit distribution $\mu$. Assume that for all $\e>0$ there exist $f \le g$ in $E$ such that
  \begin{equation}\label{eq:inv_measure_cdt}
    \mu([f,g]) \ge 1-\e. 
  \end{equation}
  Then weak synchronization holds, i.e.\ there is a $\varphi$-invariant random variable $a\in \F_0$ such that
  \begin{equation}\label{eq:convergence}
    d(\varphi_t(\t_{-t}\omega,x),a(\omega))\to 0\quad\text{for }t\to \infty,
  \end{equation}
  in probability, for all $x\in E$.
\end{thm}
\begin{proof}
The proof proceeds in several steps. In the first two steps we prove very weak synchronization, i.e.\ the existence of a $\varphi$-invariant random variable $a\in \F_0$ such that $\mu(\cdot) = \E \d_a(\cdot)$. In the last three steps we deduce \eqref{eq:convergence}.

In the following let $\pi_\cdot$ be a $\varphi$-invariant random measure associated to $\mu$ by Lemma \ref{lem:ex_stat_eq}.

\textit{Step 1}: In this step we prove that for each $\e>0$, $\d>0$ we can find $\F_0 $-measurable random sets $A(\omega)$ such that $\diam(A(\omega))\le \d$ and 
  $$\PP\big(\pi_\omega(A(\omega))\ge 1-\e\big)\ge 1-\e.$$

Let $\e >0,\d>0$ and $f,g\in E$ such that $\mu([f,g]) \ge 1-\e^2$. By Markov's inequality, 
\begin{align}\label{oben}
\PP(\pi_\omega([f,g])\ge 1-\e)&=1-\PP(\pi_\omega([f,g]^c)> \e)\\
&\ge 1-\e \nonumber.
\end{align}
For simplicity we set
$$
X_t(\omega):=\varphi_{t}(\t_{-t}\omega,f),\  Y_t(\omega):=\varphi_{t}(\t_{-t}\omega,g).
$$
By strong mixing, the laws $\LL(X_{t}),\LL(Y_{t})$ are uniformly tight for $t\ge t_0$. Hence, we may choose a compact set $K\subseteq E$ such that $\mu (K) \ge 1-\varepsilon^2$ and 
\begin{equation}\label{eq:XY-concentration}
  \PP(X_t,Y_t \in K)\ge 1-\varepsilon,\quad\forall t\ge t_0. 
\end{equation}
Again, by Markov's inequality we have that 
\begin{align}
\PP(\pi_\omega(K)\ge 1-\e)\ge 1-\e \nonumber.
\end{align}
\red{ We further observe that
$$
\pi_\omega([X_t(\omega),Y_t(\omega)]) \ge\pi_{\theta_{-t}\omega}([f,g])\quad \PP\text{-a.s.}
$$}
{\blue Since $\varphi$ is order-preserving we have
\begin{align*}
  [X_t(\omega),Y_t(\omega)]
  &= [\varphi_{t}(\t_{-t}\omega,f),\varphi_{t}(\t_{-t}\omega,g)]\\
  &\supseteq  \varphi_{t}(\t_{-t}\omega,\cdot)[f,g].
\end{align*}
Using $\varphi$-invariance of $\pi_\cdot$ we obtain
\begin{align*}
\pi_\omega([X_t(\omega),Y_t(\omega)])
&=\varphi_{t}(\t_{-t}\omega,\cdot)_*\pi_{\t_{-t}\omega}([X_t(\omega),Y_t(\omega)])\\ 
&\ge \varphi_{t}(\t_{-t}\omega,\cdot)_*\pi_{\t_{-t}\omega}(\varphi_{t}(\t_{-t}\omega,\cdot)[f,g]) \\
&\ge \pi_{\t_{-t}\omega}([f,g])\quad \PP\text{-a.s.}
\end{align*}
}
and thus, by \eqref{oben},
\begin{align}\label{oben2}
  \PP(\pi_\omega([X_t(\omega),Y_t(\omega)])\ge 1-\e)
  &\ge \PP(\pi_{\t_{-t} \omega}([f,g])\ge 1-\e)\\
  &\ge 1-\e\nonumber.
\end{align}
Hence,
\begin{align}
  \PP(\pi_\omega([X_t(\omega),Y_t(\omega)]\cap K)\ge 1-2\e\text{ and } X_t(\omega),Y_t(\omega) \in K)\ge 1-3\e, \nonumber
\end{align}
for all $t\ge t_0$.

By Proposition \ref{prop:normal} there is a function $h_K:\R_+ \to \R_+\cup \{\infty\}$ with $\lim_{t\downarrow 0} h_K(t)=0$ such that
  $$\diam([f,g]\cap K)\le h_K(d(f,g)),$$
for all $f,g\in K$. Hence, for $\omega \in \{X_t,Y_t \in K\}$ we have 
  $$\diam([X_t(\omega),Y_t(\omega)]\cap K)\le h_K(d(X_t(\omega),Y_t(\omega))).$$
By Proposition \ref{prop:monotone_convergence} we have
  $d(X_t,Y_t)\to 0$
for $t\to\infty$ in probability. Hence, with $$A_t(\omega):= [X_t(\omega),Y_t(\omega)]\cap K$$ we have
\begin{equation}\label{eq:conc_on_At}
  \PP\big( \diam(A_t(\omega))\le \d,\ \pi_\omega(A_t(\omega))\ge 1-2\e \big) \ge 1-4\e,
\end{equation}
for all $t\ge t_0=t_0(\e,\d)$. This finishes the proof of step one.

\textit{Step 2}: We show next that $\pi_\omega$ is a random Dirac measure $\PP$-a.s.. 

Let $A^n$ be as in step one with $\e,\d=2^{-n}$ and let
  $$B(\omega):=\bigcup_{n\ge 0} \bigcap_{m\ge n} A^m(\omega).$$
Then $B(\omega)$ is an $\F_0$-measurable random set. For $x,y \in B(\omega)$ we have $x,y\in \bigcap_{m\ge n} A^m(\omega)$ for all $n$ large enough. Since 
 $$\diam\left(\bigcap_{m\ge n} A^m(\omega)\right) =0$$
this implies $x=y$. Hence, $B(\omega)$ consists of at most one (random) point. Moreover, 
\begin{align*}
  \E \pi_\omega (B(\omega) )
  &= \lim_{n\to\infty} \E \pi_\omega \left(\bigcap_{m\ge n} A^m(\omega)\right) \\
  &= 1-\lim_{n\to\infty} \E \pi_\omega \left(\bigcup_{m\ge n} \left(A^m(\omega)\right)^c\right) \\
  &\ge 1- \lim_{n\to\infty} \sum_{m\ge n} 2^{-m+1} \\
  &=1.
\end{align*}
In particular, $B(\omega)=\{a(\omega)\}$ for some $\F_0$-measurable random variable $a:\Omega\to E$. In conclusion, 
\begin{equation}\label{eq:very_weak}
   \pi_\omega=\delta_{a(\omega)}\quad  \PP\text{-a.s..}   
\end{equation}
and $\varphi$-invariance of $a$ follows from  $\varphi$-invariance of $\pi_\omega$.
 
\textit{Step 3}: Let $h \in [x,y]$ for some $x\le y$ such that $\mu([x,y])>0$. We show that then 
  $$  d(a(\omega),\varphi_t(\t_{-t}\omega,h)) \to 0\quad\text{for }t\to\infty $$
in probability.

Let $\d>0$ be arbitrary, fix. Since $\mu([x,y])>0$, for each $\e>0$ small enough and each $f\le g$ with $\mu([f,g])\ge1-\e^2$ we have that $[x,y]\cap [f,g]\ne\emptyset$. Fix such $\e>0$ and $f\le g$. Further let $X_t(\omega), Y_t(\omega), K$ and $A_t(\omega)$ be defined as in step one. 

 Using Proposition \ref{prop:monotone_convergence} this yields
\begin{equation}\label{eq:weak_conv}
  d(\varphi_t(\t_{-t}\omega,h),\varphi_t(\t_{-t}\omega,f))\to 0\quad\text{for }t\to\infty 
\end{equation}
  in probability. From \eqref{eq:XY-concentration}, \eqref{eq:conc_on_At} and \eqref{eq:very_weak} we obtain that
\begin{align*}
   \PP\big( d(a(\omega),\varphi_t(\t_{-t}\omega,f))\le \d \big)
   &\ge \PP\big( \diam(A_t(\omega))\le \d,\ a(\omega)\in A_t(\omega)\big)- \PP\big(X_t \not\in K\big) \\
   &= \PP\big( \diam(A_t(\omega))\le \d,\ \pi_\omega(A_t(\omega))\ge1-2\e\big)- \PP\big(X_t \not\in K\big) \\
   &\ge 1-5\e,
\end{align*}
for all $t\ge t_0=t_0(\e,\d)$. Thus, due to \eqref{eq:weak_conv},
   $$\PP\big( d(a(\omega),\varphi_t(\t_{-t}\omega,h))\le \d \big) \ge 1-6\e,$$
for all $t\ge t_0=t_0(\e,\d)$, which finishes the proof of step 3. 

\textit{Step 4}: We prove that for each $f\le g$ with $\mu([f,g])>0$, $\d>0$ and each compact set $K\subseteq E$ we have that
 $$\lim_{t \to \infty}\PP\big([X_t,Y_t] \cap K \subseteq B_\d(a)\big) =1,$$
where $X_t(\omega)=\varphi_t(\t_{-t}\omega,f),Y_t(\omega)=\varphi_t(\t_{-t}\omega,g)$.

By strong mixing, for each $\e>0$ we may choose a compact set $K_\e$ such that $K\subseteq K_\e$ and 
$$
\PP(X_t,Y_t \in K_\e)\ge 1-\frac{\varepsilon}{4},\quad\forall t\ge t_0. 
$$
By Proposition \ref{prop:monotone_convergence}
$$
  d(X_t,Y_t)\to 0\quad\text{for }t\to\infty 
$$
  in probability. As in step one, we obtain that
$$ \PP\left(\diam([X_t,Y_t]\cap K_\e)\le \frac{\d}{2}\right)\ge 1-\frac{\e}{2},$$
for all $t\ge t_0(\e,\d)$.
Since, by step three we have $d(X_t,a)\to 0$ in probability, this implies that
\begin{align*}
  \PP\big([X_t,Y_t] \cap K \subseteq B_\d(a)\big)
  &\ge \PP\big([X_t,Y_t] \cap K_\e \subseteq B_\d(a)\big)\\
  &\ge 1-\e,
\end{align*}
for all $t\ge t_0(\e,\d)$, which finishes the proof of step four.

\textit{Step 5}: We prove that for each $x\in E$ we have
  $$ d(\varphi_t(\t_{-t}\omega,x),a(\omega))\to 0\quad\text{for }t\to \infty $$
in probability.

Fix $\delta>0$, $x\in E$. By strong mixing we may choose $K\subseteq E$ compact such that    
  $$\PP\big(\varphi_t(\cdot,x)\in K\big)\ge 1-\e, \quad\forall t\ge t_0$$
and $\mu(K)\ge 1-\e$. Furthermore, let $f\le g$ such that $\mu([f,g])\ge 1-\e$. By step four we can choose $t>0$ such that 
 $$\PP([X_t,Y_t] \cap K \subseteq B_\d(a))\ge 1-\e.$$
Hence, 
  $\PP\big([f,g] \cap \varphi_t^{-1}(\t_{-t}\omega) K \subseteq \varphi_t^{-1}(\t_{-t}\omega)B_\d(a(\omega))\big)\ge  1-\e$
and thus
  $$\PP\big([f,g]\subseteq J(\omega):=\varphi_t^{-1}(\t_{-t}\omega)(B_\d(a(\omega))\cup K^c)  \big)\ge  1-\e.$$
Since $J(\omega)$ is an open set, there is a positive random variable $b$ such that $J(\omega) \supseteq B_{b(\omega)}([f,g] \cap K)$ with probability at least $1-2\e$. Thus, choosing a constant $\b>0$ small enough we can ensure that $J(\omega)\supseteq B_\b([f,g] \cap K)$ with probability at least $1-3\e$.

By strong mixing, 
  $$\PP\big(\varphi_u(\theta_{-(u+t)}\omega,x) \in B_\b([f,g]\cap K)\big) \ge 1-3\e,$$ 
for $u$ sufficiently large and thus
 $$\PP\big(\varphi_{u+t}(\theta_{-(u+t)}\omega,x) \in K^c \cup B_\d(a(\omega))\big)\ge 1-6\e.$$ 
Due to the choice of $K$ we have 
$$
\PP(\varphi_{u+t}(\theta_{-(u+t)}\omega,x) \in K^c) \le \e,    
$$
for all $u\ge u_0$. Therefore,
$$
\liminf_{u \to \infty}\PP\big(\varphi_{u+t}(\theta_{-(u+t)}\omega,x) \in B_\d(a(\omega))\big)\ge 1-7\varepsilon.
$$
Since $\delta>0$ and $\varepsilon>0$ are arbitrary, the proof is complete.
\end{proof}

\begin{rk}\begin{enumerate}
  \item Following the same arguments as in the proof of Theorem \ref{thm:monotone} one may in fact prove weak synchronization assuming only the following weaker condition than \eqref{eq:inv_measure_cdt}: Assume that there exists a countable index set $I$ and intervals $[f_i,g_i], \,i \in I$ in $E$ with 
    $$\mu\left(\bigcup_{i\in I} [f_i,g_i] \right) = 1$$ 
  and for each pair 
  $i,j \in I$ there exists some $n \in \N$ and indices $i=i_1,i_2,...,i_n=j$ such that $[f_{i_k},g_{i_k}]\cap[f_{i_{k+1}},g_{i_{k+1}}]\neq \emptyset$ for every $k \in \{1,...,n-1\}$.
  \item If $\varphi$ is a white noise RDS then the proof of Theorem \ref{thm:monotone} can be simplified. Namely, once it has been shown that the statistical equilibrium $\pi_\omega$ is a random Dirac measure (step 2 in the proof of Theorem \ref{thm:monotone}) then \cite[Proposition 2.18]{FGS14} can be applied to obtain weak synchronization.
\end{enumerate}
\end{rk}

The following example demonstrates that the assumptions of Theorem \ref{thm:monotone}  indeed only guarantee weak synchronization and not synchronization:
\begin{ex} Consider the SDE
   $$dX_t=X_t(1-X_t)dW_t$$
  on the one-dimensional torus. Then, the associated RDS is strongly mixing with invariant measure $\mu=\d_0$ and the trivial partial order ($x \le y$ implies $x=y$) is preserved. By Theorem \ref{thm:monotone}, $\{0\}$ is a weak minimal point attractor and weak synchronization holds. However, the weak attractor (which trivially exists) is the whole torus and thus synchronization does not hold.
\end{ex}

\section{Stochastic porous media equations}\label{sec:SPME}

We consider the stochastic porous medium equation
\begin{equation}\label{eq:SPME}
  dX_t = \left(\Delta X_t^{[m]}+X_t\right)dt+dW_t,
\end{equation}
with zero Dirichlet boundary conditions on a bounded, smooth domain $\OO\subseteq \R^d$, $d\le 4$, $m > 1$ and $W$ being a trace-class Wiener process on $H^{-1} := (H_0^1)^*$ with covariance operator $Q\in L(H^{-1})$. For simplicity we use $u^{[m]}:=|u|^{m-1}u$ and we set $V=L^{m+1}(\OO)$.

We first recall that the attractor for the deterministic porous medium equation
   $$dX_t = \left(\Delta X_t^{[m]}+X_t\right)dt,$$
has infinite fractal dimension (cf.\ \cite{EZ08}). In this section, we will show that weak synchronization by noise occurs if $Q$ is non-degenerate (in a sense to be made precise below). In particular, the infinite dimensional deterministic attractor collapses into a zero dimensional random attractor if enough noise is added.

In \cite{G13,BGLR10} a continuous RDS $\varphi$ corresponding to \eqref{eq:SPME} has been constructed on $H^{-1}$, which is easily seen to be a white-noise RDS. We shall assume that the corresponding Markovian semigroup $P_t f(x):=\E f(\varphi_t(\cdot,x))$ is strongly mixing. 

{\blue
\begin{rem}
Sufficient conditions for $\varphi$ corresponding to \eqref{eq:SPME} to be strongly mixing have been given, for example, in \cite{L09}. More precisely, in \cite{L09} strong mixing was shown under the following non-degeneracy assumption for the noise: $Q^\frac{1}{2}$ is injective and
\begin{equation}\label{eq:non_deg_noise}
   \|u\|_V^{m+1} \ge c \|u\|_{Q^\frac{1}{2}}^\s\|u\|_{H^{-1}}^{m+1-\s}\quad \forall u\in V,
\end{equation}
for some $\s\ge 2$, $\s>m-1$, $c>0$, where
\begin{align*}  \|u\|_{Q^\frac{1}{2}}:= \begin{cases}
  \|y\|_{H^{-1}},\quad &Q^\frac{1}{2}y=u \\
  \infty,\quad &\text{otherwise}.   
\end{cases}\end{align*}
From \cite[Corollary 1.3]{W07} we recall the following example: Let $d=1$, $Qe_i = q_i^2 e_i$ with $\sum_{i=1}^\infty \frac{q_i^2}{\l_i}<\infty$ and $e_i$, $\l_i$ the eigenvectors and eigenvalues of $-\Delta$ with domain $(H_0^1\cap H^2)(\OO)$. If $\inf_i q_i^2 >0$ then \eqref{eq:non_deg_noise} holds for any nonnegative $\s \in (m-1,m+1]$. \end{rem}}

If the semigroup $P_t$ is strongly mixing with invariant measure $\mu$ then $\mu$ is concentrated on $V$. Indeed: By It\^o's formula we have that
  $$\|X_t\|_{H^{-1}}^2+\int_0^t \|X_r\|_V^{m+1}dr = \|x_0\|_{H^{-1}}^2+ t\textrm{tr}(Q) $$
and hence
  $$\frac{1}{t}\int_0^t P_r(\|\cdot\|_V^{m+1})(x_0)dr \le \frac{1}{t} \|x_0\|_{H^{-1}}^2+ \textrm{tr}(Q).$$
Strong mixing thus implies that $\mu$ is supported on $V$, i.e.\ $\mu(V)=1$. 

The usual partial order on $H^{-1}$ is defined by: For $x,y\in H^{-1}$ set
  $$x \le y \text{ iff } (y-x)(h)\ge 0,\quad \forall \text{nonnegative } h\in H_0^1.$$
It is not difficult to see that $\varphi$ is ``$\le$"-order-preserving {\blue (cf.\ Lemma \ref{lem:spme_order} below)}. However, it is unclear how to check \eqref{eq:inv_measure_cdt}, since bounded sets in $V$ are not necessarily contained in intervals $[f,g]_\le$. Because of this, synchronization by noise for \eqref{eq:SPME} was left as an open problem in \cite{G13}. 

The key idea here is to introduce an alternative partial order ``$\preceq$" on $H^{-1}$, that is also preserved by $\varphi$ and that is better adapted to the topology of $V$, in the sense that bounded sets in $V$ are contained in intervals $[f,g]_\preceq$: For $x,y\in H^{-1}$ we define
  $$x\preceq y \text{ iff } (-\Delta)^{-1} x \le_{H_0^1} (-\Delta)^{-1}y,$$
where the partial order ``$\le_{H_0^1}$" on $H_0^1$ is defined by: $x\le_{H_0^1} y$ iff $x(\xi) \le y(\xi)$ for a.a.\ $\xi\in \OO$.

{\blue To the best of the authors' knowledge, this partial order on $H^{-1}$ has not been previously introduced in the study of the porous medium equation. However, in \cite{DD79} porous media equations of the type
  $$\partial_t u= \Delta \beta(u)+f$$
with zero Dirichlet boundary conditions and $\beta$ a continuous, non-decreasing real function satisfying $\beta(0)=0$ have been studied by means of the ``dual" problem
\begin{equation}\label{eq:DD_v}
  \partial_t v=  -\beta(-\Delta v)+(-\Delta )^{-1}f,
\end{equation}
obtained by setting $v=(-\Delta)^{-1}u$. In \cite{DD79} a comparison principle for solutions to \eqref{eq:DD_v} was shown, which corresponds, roughly speaking, to $u$ being  ``$\preceq$" order-preserving on  $H^{-1}$.}

\begin{rem}\label{rmk:not_normal}
  The partial order ``$\preceq$" is \emph{not} normal on $H^{-1}$.
\end{rem}
\begin{proof}
  We restrict to the case $\OO=(0,2\pi+2)$. Arbitrary open, smooth domains $\OO\subseteq \R^d$ can be treated similarly and by scaling.
  
  We define
  $$\tilde f^n(x) := \begin{cases} 
              x,&\quad x\in [0,1]\\ 
              1+\sin(n(x-1)),&\quad x\in [1,2\pi+1]\\ 
       2\pi+2-x,&\quad x\in [2\pi+1,2\pi+2]
       \end{cases}$$
  and
  $$\tilde g(x) := \begin{cases} 
                  2x,&\quad x\in [0,\pi+1]\\ 
                  4\pi+4-2x,&\quad x\in [\pi+1,2\pi+2].
           \end{cases}$$
  Then $\tilde f^n, \tilde g \in H_0^1(0,2\pi+2)$, $0\le \tilde f^n \le \tilde g$ and $\|\tilde f^n\|_{H_0^1(0,2\pi+2)}^2 \sim n^2$. Since $\tilde f^n \in [0,\tilde g]_{H_0^1}$ we observe $\diam_{H_0^1} ([0,\tilde g]_{H_0^1}) \ge \|\tilde f^n\|_{H_0^1} \sim n$. Thus,
    $$\diam_{H_0^1} ([0,\tilde g]_{H_0^1}) = \infty.$$
  We note that
  {\blue
  \begin{align*}
    [0,g]_{\preceq} 
    &= \{h\in H^{-1}: \ 0 \preceq h \preceq g\} \\
    &= \{h\in H^{-1}: \ 0 \le (-\Delta)^{-1} h \le (-\Delta)^{-1} g\} \\
    &= (-\Delta)\{\td h\in H_0^1: \ 0 \le \td h \le (-\Delta)^{-1} g\} \\
    &= (-\Delta)[0,(-\Delta)^{-1} g]_{H_0^1}.
  \end{align*}
  Consequently,
  \begin{align*}  
    \diam_{H^{-1}}([0,g]_{\preceq})
    &= \sup_{x,y\in [0,g]_{\preceq}} \|x-y\|_{H^{-1}}\\
    &= \sup_{x,y\in [0,(-\Delta)^{-1} g]_{H_0^1}} \|(-\Delta)x-(-\Delta)y\|_{H^{-1}}\\
    &= \sup_{x,y\in [0,(-\Delta)^{-1} g]_{H_0^1}} \|x-y\|_{H_0^{1}}\\
    &= \diam_{H_0^1}([0,(-\Delta)^{-1}g]_{H_0^1}).
  \end{align*}}
  Hence, for $g=-\Delta \tilde g$ we obtain $\diam_{H^{-1}}([0,g]_{\preceq})=\infty$. In particular, $\preceq$ is not normal on $H^{-1}$.
\end{proof}
We next prove that ``$\preceq$" is preserved by $\varphi$:
\begin{lemma}\label{lem:spme_order}
  Let $x\preceq y$, $x,y\in H^{-1}$, then
  \begin{equation}\label{eq:SPME_comp}
     \varphi_t(\omega,x)\preceq \varphi_t(\omega,y),
  \end{equation}
  for all $t\in \R_+,\ \omega\in \Omega$. {\blue Moreover, if $x\le y$, $x,y\in H^{-1}$, then
    \begin{equation}\label{eq:SPME_comp_2}
       \varphi_t(\omega,x)\le \varphi_t(\omega,y),
    \end{equation}
    for all $t\in \R_+,\ \omega\in \Omega$.}
\end{lemma}
\begin{proof} {\blue We first prove \eqref{eq:SPME_comp}: }
  For the proof it is enough to consider a fixed, arbitrary interval $[0,T] \subseteq \R_+$. We first briefly recall the construction of $\varphi$ given in \cite{G13}: In \cite[Theorem 3.2, iii]{G13}, first a strictly stationary solution $Z$ to
    $$dZ = \Delta Z^{[m]}dt+dW_t$$
  is constructed, satisfying {\blue $Z_t(\omega)=Z_0(\t_t \omega)$ for all $\omega \in \Omega$, $t\in \R$ and $Z_\cdot(\omega) \in L^{m+1}_{loc}(\R,V)\cap C(\R;H^{-1})$ for all $\omega \in \Omega$}. Then it is shown that the transformed equation (informally arising by the transformation $Y:=X-Z$)
  \begin{equation}\label{eq:tranf_SPME}\begin{split}
    \frac{d}{dt} Y_t&=\Delta (Y_t+Z_t)^{[m]}+Y_t+Z_t-\Delta Z^{[m]}_t \\
    Y_0 &= x-Z_0
  \end{split}  \end{equation}
  has a unique solution $Y$ for each fixed $\omega \in \Omega$. In the following we let $\omega \in \Omega$ be arbitrary, fixed and suppress the $\omega$-dependency in the notation. The RDS $\varphi$ is then defined by
    $$\varphi_t(\omega,x):=Y_t(\omega) + Z_t(\omega),\quad t\in\R_+,\omega\in\Omega,x\in H^{-1}.$$
  For the proof of \eqref{eq:SPME_comp} it is thus enough to consider $Y$.
  
  Let $J^\e := (1-\e \Delta)^{-1}$ be the resolvent of $-\Delta$ on $H^{-1}$. Since $J^\e$ and $(-\Delta)^{-1}$ commute, $J^\e$ is ``$\preceq$"-order-preserving. Moreover, $J^\e: H^{-1} \to H^{1}_0$ and $J^\e: H_0^1 \cap H^m \to H_0^1 \cap H^{m+2}$ for all $m\in \N$. We further note $\|J^\e x\|_{H^{-1}} \le \|x\|_{H^{-1}}$. By iterating $J^\e$, for each $l\in \N$ we may thus construct linear operators $G^{\e,l}: H^{-1} \to H^{2l-1}\cap H_0^1$. Note $G^{\e,l}x \to x$ in $H^{-1}$ for $\e\to 0$ and each fixed $l\in \N$.
    
  We further consider an approximation $Z^n$ smooth in time and space, such that $Z^n \to Z$ in $L^{m+1}([0,T];V)\cap C([0,T];H^{-1})$ and the corresponding unique solution $Y^n$ to
  \begin{align*}
    \frac{d}{dt} Y^n_t
    &=\Delta (Y^n_t+Z^n_t)^{[m]}+Y^n_t+Z^n_t-\Delta (Z^n_t)^{[m]} \\
    Y^n_0 &= G^{\frac{1}{n},l} Y_0,
  \end{align*}
  where $l$ is chosen large enough to justify the following arguments. We then define the transformation $u^n := Y^n + Z^n$ and observe that $u^n$ is the unique solution to
 {\blue \begin{equation}\label{eq:un}\begin{split}
     \frac{d}{dt} u^n_t&=\Delta (u^n_t)^{[m]}+u^n+ \Delta(Z^n_t)^{[m]} - \frac{d}{dt} Z^n_t \\
     u^n_0 &= Y_0^n + Z_0^n 
     = G^{\frac{1}{n},l} x- G^{\frac{1}{n},l}Z_0 + Z_0^n.
  \end{split}\end{equation}}
  Since, $ \Delta(Z^n_t)^{[m]} - \frac{d}{dt} Z^n_t$ is smooth, we may apply \cite[Lemma 1,\ cf. also Corollary 1]{DD79}, to obtain $ u^{1,n}_t \preceq u^{2,n}_t$, {\blue where $u^{1,n}_t$ and $u^{2,n}_t$ correspond to the solutions to \eqref{eq:un} with initial conditions  $ G^{\frac{1}{n},l} x- G^{\frac{1}{n},l}Z_0 + Z_0^n$ and $G^{\frac{1}{n},l} y- G^{\frac{1}{n},l}Z_0 + Z_0^n$ respectively}, and thus also
  \begin{equation}\label{eq:Yn-ineq}
      Y^{1,n}_t \preceq Y^{2,n}_t,\quad\forall t\ge0.
  \end{equation}
  {\blue The idea behind \cite[Lemma 1]{DD79} is to consider the ``dual" problem obtained by setting $v_n:=(-\Delta)^{-1}u_n$ which solves the fully nonlinear PDE
    $$     \frac{d}{dt} v^n_t= (\Delta v^n_t)^{[m]}+v^n -(Z^n_t)^{[m]} - (-\Delta)^{-1}\frac{d}{dt} Z^n_t.$$
  In \cite{DD79} it is then shown that this ``dual" problem satisfies a comparison principle. 
  }
  
  We next need to prove convergence of the chosen approximation. Using standard bounds for the porous medium operator $\Delta u^{[m]}$ on the Gelfand triple $V\subseteq H^{-1} \subseteq V^*$ (cf.\ e.g.\ \cite[Example 4.1.11]{PR07}), we observe
  \begin{align*}
     \frac{d}{dt} \|Y^n_t\|_{H^{-1}}^2
     =&\ _{V^*}\<\Delta (Y^n_t+Z^n_t)^{[m]}+Y^n_t+Z^n_t,Y^n_t\>_V-\ _{V^*}\<\Delta (Z^n_t)^{[m]},Y^n_t\>_V\\
     \le&-\|Y^n_t+Z^n_t\|_{m+1}^{m+1}+\|Y^n_t\|_{H^{-1}}^2+\ _{V^*}\<Z^n_t,Y^n_t\>_V
     +\|\Delta (Y^n_t+Z^n_t)^{[m]}\|_{V^*}\|Z^n_t\|_V\\&-\|\Delta (Z^n_t)^{[m]}\|_{V^*}\|Y^n_t\|_V\\
     \le&-c\|Y^n_t\|_V^{m+1}+C\|Z^n_t\|_{V}^{m+1}+C\|Y^n_t\|_{H^{-1}}^2+C\|Z^n_t\|_{H^{-1}}^2+C\|Z^n_t\|_V^{m+1}\\
     &+C_\e\|Z^n_t\|_{V}^{m+1}+\e \|Y^n_t\|_V^{m+1} \\
     =& -(c-\e)\|Y^n_t\|_V^{m+1}+C\|Y^n_t\|_{H^{-1}}^2+C\|Z^n_t\|_{H^{-1}}^2+C_\e\|Z^n_t\|_V^{m+1},
  \end{align*}
  for some constants $c,C,C_\e>0$ and all $\e>0$. Choosing $\e$ small enough and using Gronwall's Lemma yields
  \begin{align*}
     \sup_{t\in [0,T]} \|Y^n_t\|_{H^{-1}}^2 +c \int_0^T \|Y^n_t\|_V^{m+1}dt \le C,
  \end{align*}
  for some uniform constants $c,C>0$. Hence, also $\frac{d}{dt}Y^n_t \in L^{m+1}([0,T];V^*)$ with uniform bounds and by the Aubin-Lions compactness Lemma we obtain the existence of a subsequence (again denoted by $Y^n$) such that
  \begin{align*}
     Y^n &\to \tilde Y\quad \text{in } C([0,T];H^{-1}) \\
     Y^n &\rightharpoonup \tilde Y\quad \text{in } L^{m+1}([0,T];V).
  \end{align*}
  It is then not difficult to identify $\tilde Y$ as a variational solution to \eqref{eq:tranf_SPME} and uniqueness implies $\tilde{Y} = Y$. Since the partial order ``$\preceq$" is closed on $H^{-1}$, from \eqref{eq:Yn-ineq} we obtain $ Y^{1}_t \preceq Y^{2}_t$ which implies the claim.
  
  {\blue To prove \eqref{eq:SPME_comp_2} we may proceed analogously. Indeed, following \cite{B72} and \cite[equation (4)]{DD79} we obtain $ u^{1,n}_t \le u^{2,n}_t$ and thus also $Y^{1,n}_t \le Y^{2,n}_t$ for all $t\ge0$, replacing \eqref{eq:Yn-ineq} above. This implies \eqref{eq:SPME_comp_2} following the same arguments as for \eqref{eq:SPME_comp}.}
\end{proof}

\begin{thm} Assume that the RDS $\varphi$ associated to \eqref{eq:SPME} is strongly mixing. Then, $\varphi$ has a singleton weak point attractor $A$, i.e.\ weak synchronization holds. Moreover, $A$ attracts all sets $K \subseteq H^{-1}$ contained in ``$\le$"-intervals, i.e.\ all $K\subseteq [f,g]_\le$ for some $f,g\in H^{-1}$. 
\end{thm}
\begin{proof}
  As noted above $\varphi$ is a ``$\preceq$"-order-preserving, white noise RDS on $H^{-1}$ and the invariant measure $\mu$ is concentrated on $V$. It remains to check \eqref{eq:inv_measure_cdt} with respect to the partial order ``$\preceq$".

  We first observe that $W^{2,m+1} \hookrightarrow C^0$ if $2-\frac{d}{m+1}>0$, or equivalently $2(m+1)>d$, which is satisfied since $m>1,d\le 4$. Recall $\mu(V)=1$. We now consider $\tilde \mu := (-\Delta)^{-1}_* \mu$ on $W^{2,m+1}$, i.e.\ the push-forward of $\mu$ under $(-\Delta)^{-1}$, and observe
    $$\tilde\mu(W^{2,m+1})=1.$$
  Hence, we can find $\tilde f_n = (-\Delta)^{-1} f_n \le (-\Delta)^{-1} g_n = \tilde g_n $ with $f_n,g_n \in V$ such that $\tilde\mu([\tilde f_n,\tilde g_n]_{W^{2,m+1}})\ge 1-2^{-n}$. Thus, 
  \begin{align*}
    \mu([ f_n, g_n]_{H^{-1};\preceq})
    &\ge \mu([ f_n, g_n]_{L^{m+1};\preceq}) \\
    &= \tilde \mu([(-\Delta)^{-1} f_n, (-\Delta)^{-1} g_n]_{W^{2,m+1}})\\
    &\ge 1-2^{-n}.
  \end{align*}
  Theorem \ref{thm:monotone} concludes the proof of weak synchronization.
  
  Let now $K\subseteq [f,g]_\le$. Since ``$\le$" is a normal partial order on $H^{-1}$ we have $$\diam([f,g]_\le)\le h(\|f-g\|_{H^{-1}}).$$ Since $A$ is a singleton weak point attractor, this implies $$\diam([\varphi_t(\cdot,f) ,\varphi_t(\cdot,g)]_\le)\to 0$$ and thus $\diam(\varphi_t(\cdot,K))\to 0$ in probability, which finishes the proof.
\end{proof}

We note that in general it is not true that $\varphi_t(\omega,x)$ takes values in $V$ if $x\in V$. In order to show such an invariance property additional regularity of $W$ would be required. In contrast, the invariant measure $\mu$ is always supported on $V$ as long as $W$ is a trace-class Wiener process in $H$. At this point the generalization put forward in Theorem \ref{thm:monotone} is crucial, since condition \eqref{eq:inv_measure_cdt} only requires $\mu$ to be ``nicely" supported, rather than the partial order ``$\preceq$" to be admissible, as it had to be assumed in \cite{CS04}. Since $\preceq$ is neither admissible nor normal, the results from \cite{CS04} cannot be used in the case of stochastic porous media equations.

\section{Further Examples}\label{sec:examples} 

\subsection{Stochastic differential equations driven by fractional Brownian motion}
We consider one-dimensional stochastic differential equations of the type
\begin{equation}\begin{split}\label{eq:sde-fbm}
  dX^x_t &= b(X^x_t)dt+dB^H_t\\
  X^x_0 &= x\in \R,
\end{split}\end{equation}
where $B^H$ is a two-sided fractional Brownian motion with Hurst index $H\in (0,1)$. For example, $B^H$ can be constructed by
\begin{equation}\label{eq:FBM}
  B^H_t = \a_H \int_{-\infty}^0(-r)^\frac{H}{2}(dW_{r+t}-dW_r),
\end{equation}
where $W$ is a one-dimensional Brownian motion and $\a_H$ is an appropriately chosen constant (cf.\ e.g.\ \cite{H05}). From \eqref{eq:FBM} we can read-off that fractional Brownian motion has strictly stationary increments, in the sense that
\begin{equation}\label{eq:FBM_shift}
  B^H_{t+s}(\omega)-B^H_{t}(\omega)= B^H_t(\theta_s\omega),
\end{equation}
where $\t$ is the usual Wiener shift.

We further assume that there are constants $c>0,C,N\ge 0$ such that
  $$(b(x)-b(y))(x-y)\le \min(C-c|x-y|^2,C |x-y|^2),$$
for all $x,y\in \R$ and 
  $$|b(x)|+|b'(x)|\le C(1+|x|)^N,$$
for all $x\in \R$. If $H\ge \frac{1}{2}$ we further assume that $b'$ is globally bounded. 

In order to construct the associated RDS one considers the transformation $Y^x_t = X^x_t -B^H_t$ satisfying
\begin{equation*}\begin{split}
  dY^x_t &= b(Y^x_t+B^H_t)dt\\
  Y^x_0 &= x\in \R,
\end{split}\end{equation*}
which is easily seen to have a unique solution. Thus, in view of \eqref{eq:FBM_shift},
   $$\varphi_t(\omega,x):=Y^x_t(\omega)+B^H_t(\omega)$$
defines a continuous RDS. By uniqueness of solutions $\varphi$ is order-preserving on $\R$.

Following the setup put forward in \cite{H05} the stochastic dynamical system\footnote{For the notion of a stochastic dynamical system cf.\ \cite{H05}.} $\td\varphi$ associated to \eqref{eq:sde-fbm} is a weak solution to  \eqref{eq:sde-fbm}. Since also $\varphi$ is a weak solution, by weak uniqueness we have $\LL(\td\varphi_t(\cdot,x))=\LL(\varphi_t(\cdot,x))$. Moreover, by \cite[Theorem 6.1]{H05} there is a probability measure $\mu$ on $\R$ such that
  $$\LL(\td\varphi_t(\cdot,x))\to\mu\quad\text{for }t\to\infty$$
in total variation norm. Hence, $\varphi$ is strongly mixing. We conclude
\begin{ex}
  The RDS $\varphi$ associated to \eqref{eq:sde-fbm} satisfies weak synchronization.
\end{ex}

\subsection{Stochastic differential inclusions and reflected diffusions}
We consider stochastic differential inclusions of the type
\begin{equation}\begin{split}\label{eq:mv-sde}
  dX^x_t + \partial \eta(X^x_t)dt&\ni b(X^x_t)dt+dW_t\\
  X^x_0 &= x\in \R^d,
\end{split}\end{equation}
where $\partial \eta$ is the subdifferential of a convex, lower semicontinuous (lsc), proper function $\eta:\R^d\to \R\cup\{+\infty\}$ with domain $\mathrm{dom}(\eta)$, $W$ is a standard Brownian motion on $\R^d$ and $b$ is globally Lipschitz continuous. We assume that
  $$\mathrm{int}(D(\partial\eta))\ne\emptyset.$$
By \cite{C98}, for each $x \in \overline{D(\partial \eta)}=\overline{\mathrm{dom}(\eta)}$ there exists a unique solution $X^x \in C([0,T];\R^d)$ to \eqref{eq:mv-sde} taking values in $\overline{D(\partial\eta)}$. Since the construction in \cite{C98} is path-wise, i.e.\ existence and uniqueness for \eqref{eq:mv-sde} is proven for $W$ replaced by an arbitrary continuous path starting at $0$, 
  $$\varphi_t(\omega,x):=X_t^x(\omega),\quad t\in\R_+,\omega\in \Omega, x\in \R^d,$$
defines a continuous RDS on $E:=\overline{D(\partial\eta)}$. We note that in general $\overline{D(\partial\eta)}$ is a proper subset of $\R^d$, indeed:
\begin{ex}[Reflected diffusions]
   Let $\overline D\subseteq \R^d$ be a nonempty, closed, convex set in $\R^d$. Let $$\eta(x)=I_{\overline D}(x)=\begin{cases} 0\quad &\text{if } x\in \overline D \\ +\infty\quad &\text{otherwise}. \end{cases}$$
   Then \eqref{eq:mv-sde} corresponds to 
       $$dX_t=b(X_t)dt+dW_t,$$
   with normal reflection on $\partial\overline D$.
\end{ex}

If $\mathrm{dom}(\eta)$ is a bounded set and $b\in C^2(\R^d)$ with bounded derivatives, then $\varphi$ is strongly mixing by \cite{CJ98}.

In order to apply Theorem \ref{thm:monotone} we need $\varphi$ to be order-preserving. Therefore, we shall restrict to $d=1$ henceforth. Uniqueness of solutions to \eqref{eq:mv-sde} implies that for $x,y\in \overline{D(\partial \eta)}$ with $x\le y$ we have
  $$X^x_t(\omega) \le X_t^y(\omega)\quad \forall t\in \R_+,\ \omega\in\Omega.$$
Thus, $\varphi$ is order-preserving. We note that in general $x\mapsto \varphi_t(x,\omega)$ is not one-to-one. In particular, the strong order $x<y$ is not necessarily preserved under $\varphi$.

An application of Theorem \ref{thm:monotone} yields
\begin{ex}
  Assume that $\eta:\R\to\R$ is a convex, lsc, proper function with bounded domain and $b\in C^2(\R)$ with bounded derivatives. Then, the RDS $\varphi$ corresponding to \eqref{eq:mv-sde} satisfies weak synchronization. 
\end{ex}

\subsection{SPDE with two reflecting walls}
We consider the following SPDE with two reflecting walls
\begin{equation}\label{eq:reflected_SPDE}\begin{split}
   dX_t &=\partial_x^2 X_tdt +f(X_t)dt+d W_t\ \text{on }S^1\\
   X_0 &= x \in C(S^1) \\
   h^1 &\le X_t \le h^2\quad\forall t\in[0,T],
\end{split}\end{equation}
where $S^1$ is the one-dimensional sphere, $dW_t$ denotes space-time white noise on $S^1\times[0,T]$ and $h^1,h^2 \in C(S^1)$ satisfy $h^1<h^2$, $h^i \in H^2(S^1)$, $i=1,2$. We further assume that $f$ is Lipschitz continuous and set
  $$E:=\{h \in C(S^1):\ h^1 \le h \le h^2\}.$$

In order to construct an associated RDS we consider the Ornstein-Uhlenbeck process $Z$ corresponding to
  $$dZ_t = \partial_x^2 Z_t dt+dW_t,$$
given by $Z_t = \int_{0}^t e^{\partial_x^2 (t-s)}dW_s$ and the transformed PDE
\begin{equation}\label{eq:reflected_SPDE_2}\begin{split}
   \partial_t Y_t&= \partial_x^2 Y_t +f(Y_t{\blue +}Z_t)\\
   Y_0 &= x \in C(S^1) \\
   h^1-Z_t &\le Y_t \le h^2-Z_t\quad\forall t\in [0,T].
\end{split}\end{equation}
Well-posedness of \eqref{eq:reflected_SPDE_2} can be shown as in \cite{O06} for every $x \in E$. Uniqueness for \eqref{eq:reflected_SPDE_2} then implies that 
  $$\varphi_t(\omega,x):=Y_t(\omega)+Z_t(\omega)$$
defines a continuous RDS on $E$. Moreover, following \cite[Lemma 2.6]{O06} we have comparison, i.e.\ if $x,y \in E$ with $x \le y$ then $\varphi_t(\omega,x) \le \varphi_t(\omega,y)$. Hence, $\varphi$ is order-preserving on $E$. It remains to observe that by the coupling argument used in \cite[Theorem 3.1]{YZ14} $\varphi$ is strongly mixing. An application of Theorem \ref{thm:monotone} yields
\begin{ex}
  The RDS $\varphi$ corresponding to  \eqref{eq:reflected_SPDE} satisfies weak synchronization.
\end{ex}

\appendix

\section{Background on random dynamical systems}\label{sec:RDS}

Let $(E,d)$ be a Polish space, that is, a topological space homeomorphic to a complete, separable metric space, endowed with Borel $\sigma$-algebra $\EE$. Further, let $\left(  \Omega,\F,\PP, \t\right)  $ be a {\em metric dynamical system}, that is, 
$(\Omega,\F,\PP)$ is a probability space (not necessarily complete) and  $\t:=\left(  \theta_{t}\right)  _{t\in\R}$ 
is a group of jointly measurable maps on $\left(  \Omega,\F,\PP\right)  $ that leaves $\PP$
invariant.

We say that a map $\varphi: \R_+ \times \Omega \times E \rightarrow E$ is a {\em perfect cocycle} if $\varphi$ is measurable, $\varphi_{0}(  \omega,x)  =x$ and $\varphi_{t+s}\left(
\omega,x\right)  =\varphi_{t}\left(  \theta_{s}\omega,\varphi_{s}\left(
\omega,x\right)  \right)  $ for all $x\in E$, $t,s\geq0$, $\omega\in\Omega$. 
We will assume that $\varphi_{s}(\w,\cdot)$ is continuous for each $s \ge 0$ and $\omega \in \Omega$. 
The collection $(\Omega, \F,\PP,\t,\varphi)$ is then said to be a 
{\em random dynamical system} (RDS), see \cite{A98} for a comprehensive treatment. Given an RDS $(\Omega, \F,\PP,\t,\varphi)$ we define the {\em skew-product} flow $\Theta$ on $\Omega\times E$ by $\Theta_t(\omega,x)=(\t_t\omega,\varphi_t(\omega,x))$. 

Let $E$ be a Polish space with closed partial order ``$\le$" (cf.\ \eqref{eq:M}) and $(\Omega, \F,\PP,\t,\varphi)$ an RDS on $E$. Then $\varphi$ is said to be ``$\le$"-order-preserving if $\varphi_t(\omega,x)\le \varphi_t(\omega,y)$ for all $x,y\in E$, $x\le y$ and all $t\ge 0$, $\omega\in \Omega$.

Given an RDS  $\varphi$ we define the two-parameter filtration $\FF=(\F_{s,t})_{-\infty < s \le t < \infty}$ of sub$-\sigma$ algebras
of $\F$ given by $\F_{s,t}=\s\{\varphi_h(\theta_s\omega):\ h\in [0,t-s]\}$. It follows that $\theta_r^{-1}(\F_{s,t})=\F_{s+r,t+r}$ for all $r,s,t$ . For each $t \in \R$, let $\F_t$ be the smallest $\sigma$-algebra containing all $\F_{s,t}$, $s \le t$  
and let $\F_{t,\infty}$ be the smallest $\sigma$-algebra containing all $\F_{t,u}$, $t \le u$. If  $\F_{s,t}$ and $\F_{u,v}$ are independent for all $s \le t \le u \le v$, we call  $(\Omega, \F,\FF,\PP,\t,\varphi)$  a 
{\em white noise (filtered) random dynamical system}. 

An invariant measure for an RDS $\varphi$ is a probability measure on $\Omega\times E$ with marginal $\PP$ on $\Omega$ that is invariant under $\Theta_t$ for $t\ge 0$. 
For each probability measure $\pi$ on $\Omega\times E$ with marginal $\PP$ on $\Omega$ there is a unique disintegration $\omega\mapsto \pi_\omega$ and a random probability measure $\pi_\omega$ is an 
invariant measure for $\varphi$ iff $\varphi_t(\omega)_*\pi_\omega = \pi_{\t_t\omega}$ for all $t\ge 0$, almost all $\omega\in \Omega$ (where the $\PP$-zero set may depend on $t$). Here $\varphi_t(\omega)_*\pi_\omega$ denotes the push-forward of $\pi_\omega$ under $\varphi_t(\omega)$. An invariant measure $\pi_\omega$ is said to be a Markov measure, if $\omega\mapsto \pi_\omega$ is measurable with respect to the past $\F_0$.
In case of a white noise RDS $\varphi$ we may define the associated Markovian semigroup by
  $$ P_t f(x):=\E f(\varphi_t(\cdot,x)), $$
for $f$ being measurable, bounded. There is a one-to-one correspondence between invariant measures for $P_t$ and Markov invariant measures for $\varphi$ (cf.\ \cite{C91}): {\blue If $\mu$ is $P_t$-invariant, then for every sequence $t_k \to \infty$ the weak$^*$ limit
\begin{equation}\label{eq:inv_meas}
   \pi_\omega := \lim_{k\to\infty}\varphi_{t_k}(\t_{-{t_k}}\omega)_*\mu 
\end{equation}
exists $\PP$-a.s. and it is a Markov invariant measure for  $\varphi$. In addition, $\pi_\cdot$ does not depend on the chosen sequence $t_k$, $\PP$-a.s. Vice versa, $\mu:=\E\pi_\omega$ defines an invariant measure for $P_t$.}

A Markovian semigroup $P_t$ with ergodic measure $\mu$ is said to be \textit{strongly mixing} if 
 $$P_t f(x) \to \int_E f(y) d\mu(y)\quad \text{for } t\to\infty$$ 
for each continuous, bounded $f$ and all $x\in E$. Similarly, we say that an RDS $\varphi$ (not necessarily a white noise RDS) is strongly mixing if the laws of $\varphi_t(\cdot,x)$ converge weakly$^*$ to a probability measure $\mu$ for $t\to\infty$ for all $x\in E$.

\begin{defn}
  A family $\{D(\w)\}_{\w \in \Omega}$ of non-empty subsets of $E$ is said to be
  \begin{enumerate}
  \item  a random closed (resp. compact) set if it is $\PP$-a.s.\ closed (resp. compact) and $\w \mapsto d(x,D(\w))$ is $\mathcal{F}$-measurable for each $x \in E$. In this case we also call $D$, $\mathcal{F}$-measurable.
  \item $\varphi$-invariant, if for all $t\ge0$
  		$$\varphi_t(\w,D(\w))=D(\t_t\w),$$
  		for almost all $\w\in \Omega$.
  \end{enumerate}
  
\end{defn}

Next, we recall the definition of a pullback attractor and a weak (random) attractor (cf. \cite{CF94,O99}). 
 
\begin{defn}
Let $(\Omega, \F,\PP,\t,\varphi)$ be an RDS. A random, compact set $A$ is called 
a {\em pullback attractor}, if
\begin{enumerate}
\item $A$ is $\varphi$-invariant, and 
\item for every compact set $B$ in $E$, we have
$$
\lim_{t \to \infty}\sup_{x \in B} d(\varphi_t(\theta_{-t}\omega,x),A(\omega))=0, \mbox{ almost surely}.
$$
\end{enumerate}
The map $A$ is called a {\em weak attractor}, if it satisfies the properties above with almost sure convergence replaced by 
convergence in probability in (2). It is called a {\em (weak) point attractor}, if it satisfies the properties above with compact sets $B$ replaced by single points in (2).

A (weak) point attractor is said to be {\em minimal} if it is contained in each (weak) point attractor.
\end{defn}

Clearly, every pullback attractor is a weak attractor but the converse is not true (see e.g.~\cite{S02} for examples). Weak attractors are unique (cf. \cite[Lemma 1.3]{FGS14}).

\begin{defn}
Let $(\Omega, \F,\PP,\t,\varphi)$ be an RDS. We say that (weak) synchronization occurs, if there exists a weak (point) attractor consisting of a single random point $\PP$-a.e.. 
\end{defn}

{\blue
\subsection*{Acknowledgement}
We thank the anonymous referees for several suggestions that helped to improve the presentation of the manuscript.}

\bibliographystyle{plain}
\bibliography{synchronization-refs}

\end{document}